\documentclass[12pt,twoside,leqno]{article}
\usepackage{amsmath}
\usepackage{amssymb}
\usepackage{amsxtra}
\usepackage{amscd}
\usepackage{amsthm}
\usepackage[mathscr]{eucal}
\usepackage{xr}
\usepackage{color}

\swapnumbers

\theoremstyle{plain}
\newtheorem{thm}[subsection]{Theorem}
\newtheorem{prop}[subsection]{Proposition}

\theoremstyle{definition}

\newtheorem{rem}[subsection]{Remark}
\newtheorem{para}[subsection]{}
\newenvironment{pf}{\proof[\proofname]}{\endproof}
\newenvironment{pf*}[1]{\proof[#1]}{\endproof}

\newcommand{\Z}{{\mathbb{Z}}}

\newcommand{\Spec}{\operatorname{Spec}}

\newcommand{\fs}{{\rm fs \ }}

\newcommand{\Gm}{\mathbb{G}_m} %
\newcommand{\Gmlog}{\mathbb{G}_{m,\log}}

\newcommand{\Ext}{\operatorname{Ext}}

\renewcommand{\bar}[1]{\overline{#1}}
\def \et {\mathrm {\acute{e}t}}

\newcommand{\cl}{Claim}
\theoremstyle{definition}

\newtheorem*{clm*}{\cl}

\theoremstyle{plain}

\def \fs {\mathrm {fs}}
\def \ket {\mathrm {k\acute{e}t}}
\def \kfl {\mathrm {kfl}}
\newcommand{\gr}{\mathrm {gr}}

\def \overc#1{\overset {\lower 0.3ex \hbox{${\;}_{\circ}$}}{#1}}

\let\refsave=\ref
\def\ref#1{\textup{\refsave{#1}}}

\newcommand{\upc}{\overset{\circ}\to}
\newcommand\Cal{\mathcal}
\newcommand\define{\newcommand}
\renewcommand\bold{\Bbb}

\define\bZ{\bold Z}

\define\bQ{\bold Q}

\define{\cS}{\Cal S}
\define{\Lie}{\mathrm{Lie}\,} %
\define{\coLie}{\mathrm{coLie}\,} %
\define{\cH}{\Cal H}
\define{\cExt}{{\Cal E}xt}
\define{\cHom}{{\Cal H}om}
\define{\cO}{\Cal O}
\define{\an}{\mathrm{an}}

\def \upcf {\overset {\lower 0.3ex \hbox{${\;}_{\circ}$}} f}
\def \upcp {\overset {\lower 0.3ex \hbox{${\;}_{\circ}$}} p}
\def \upc#1{\overset {\lower 0.3ex \hbox{${\;}_{\circ}$}}{#1}}

\newcommand{\Sig}{\Sigma}

\begin{document}
\title
{Duals of log abelian varieties}
\author{Takeshi Kajiwara, Kazuya Kato, and Chikara Nakayama}
\date{}
\maketitle
\setlength{\baselineskip}{1.0\baselineskip}
\begin{abstract}
  We prove the existence of the dual of a polarizable log abelian variety.
\end{abstract}

\section*{Contents}

\noindent \S \ref{s:A[n]}. Torsion points of weak log abelian varieties

\noindent \S \ref{s:A/H}. The Quotient of a log abelian variety by a finite group scheme

\noindent \S \ref{s:dual}. Dual is a log abelian variety

\noindent \S \ref{s:duality}. $A^{**}=A$. 

\medskip

\renewcommand{\thefootnote}{\fnsymbol{footnote}}
\footnote[0]{Primary 14K10; 
Secondary 14J10, 14D06} 

\section*{Introduction}
  It is well-known that the dual of an abelian scheme always exists as an abelian scheme. 
  By analogy, we expect that the dual of a log abelian variety always exists as a log abelian variety.  
  But it has been known only in the constant degeneration case so far. 
  In this article, we prove it for a polarizable log abelian variety.
  The assumption of the polarizability is essential to our proof, but the result is enough for most applications (f.e., moduli problems), where we only treat polarizable ones. 
  
  In Section \ref{s:A[n]}, as a preparation, we discuss the structure of the torsion points of log abelian varieties. 
  In Section \ref{s:A/H}, we prove a general theorem which says that 
the quotient of a log abelian variety by a finite log flat subgroup log scheme is a log abelian variety.  
  In Section \ref{s:dual}, applying this theorem, we prove that the existence of the dual.
  In Section \ref{s:duality}, we prove that the dual of the dual coincides with the original one. 

\smallskip

\noindent {\sc Acknowledgments.} 
The first author is partially supported by JSPS, Kakenhi (C) 20K03555 and (C) 25K06951. 
The second author is partially  supported by NFS grants DMS 1601861 and DMS 2001182.
The third author is partially supported by JSPS, Kakenhi (C) 21K03199 and (C) 26K06723.

\section{Torsion points of weak log abelian varieties}
\label{s:A[n]}
  In this section, we discuss the structure of the torsion points of weak log abelian varieties. 
  See \cite{KKN4} Definition 1.6 for the definition of weak log abelian varieties. 
  Though we use the results in this section only for log abelian varieties later, we discuss them here for weak log abelian varieties because the proofs are the same. 

  First we note the following fact. 

\begin{prop}
\label{p:kfl}
  A weak log abelian variety over an fs log scheme is a kfl 
(= kummer log flat) 
sheaf.
\end{prop}

\begin{pf}  Theorem 1.1 of \cite{Z} is the case where the base is locally noetherian, to which the general case is easily reduced by \cite{KKN4} Proposition 9.2 (2).
\end{pf}

  In the following, the filtration $(W_w)_{w \in \bZ}$ of an object, the quotients $W_w/W_{w'}$ ($w'\leq w$), and exact sequences  are considered in the category of sheaves of 
abelian groups on the kfl site.

\begin{para}\label{W1} In \ref{W1} and \ref{W2}, let $A$ be a weak log abelian variety with constant degeneration. Let $n\geq 1$.
  Let $G$ be the semiabelian part of $A$ whose torus part is denoted by $T$. 
  Let $B=G/T$ be the quotient abelian scheme.
  Let $X=\cHom(T,\Gm)$ and $X \times Y \to \Gmlog/\Gm$ the induced admissible pairing.  

  Then we have the following  increasing filtration $W$ on $A[n]$ by finite kfl subgroup fs log schemes 
(this means subgroup sheaves represented by a finite kfl fs log scheme)
called the {\it weight filtration}. 

$W_w=A[n]$ for $w\geq 0$; $W_{-1}=G[n]$; $W_{-2}= T[n]$; $W_w=0$ for $w\leq -3$.

We have 
$$\gr^W_0=Y/nY, \quad \gr^W_{-1}=B[n], \quad \gr^W_{-2}=T[n].$$

\end{para}

\begin{para}\label{W2} Let $A^*$ be the dual weak log abelian variety of $A$ which we defined in the case of constant degeneration in \cite{KKN5} 1.2. 
Then $A^*[n]$ is the Cartier dual of $A[n]$. 
  Furthermore, we have: 

For $w'\leq w$,  $W_w/W_{w'}$ of $A[n]$ and  $W_{-3-w'}/W_{-3-w} $ of $A^*[n]$ are the Cartier duals of each other. 

In particular, $A[n]/W_{-2}$ is the Cartier dual of $G^*[n]$ and hence is classical finite flat, where $G^*$ is the semiabelian part of $A^*$. 
\end{para}

\begin{para}\label{cat}
  We review the definition of some categories of sheaves of abelian groups introduced in \cite{Kato:logD}. 

  For an fs log scheme $S$, 
  let $(\mathrm{fin}/S)_{\mathrm r}$ be the category of sheaves of abelian groups on the kfl site of $S$ 
represented by a finite kfl group fs log scheme over $S$.
  Let $(\mathrm{fin}/S)_{\mathrm d}$ be the full subcategory of $(\mathrm{fin}/S)_{\mathrm r}$ consisting of all objects whose Cartier dual also belongs to $(\mathrm{fin}/S)_{\mathrm r}$.
  Let $(\mathrm{fin}/S)_{\mathrm c}$ be the full subcategory of $(\mathrm{fin}/S)_{\mathrm r}$ consisting of all objects which is classical finite flat over $S$. 
  Let $(\mathrm{fin}/S)_{\mathrm e}$ (resp.\ $(\mathrm{fin}/S)_{\mathrm f}$) be the category of sheaves of abelian groups on the kfl site of $S$ which belongs to 
$(\mathrm{fin})_{\mathrm r}$ k\'et (= kummer log \'etale) (resp.\ kfl) locally. 
  We have inclusions
$$(\mathrm{fin}/S)_{\mathrm c} \subset (\mathrm{fin}/S)_{\mathrm d}\subset (\mathrm{fin}/S)_{\mathrm r}\subset (\mathrm{fin}/S)_{\mathrm e}\subset (\mathrm{fin}/S)_{\mathrm f}.$$

  Concerning these categories, the following facts are basic. 

(1) An object of $(\mathrm{fin}/S)_{\mathrm f}$ killed by an integer which is invertible on $S$ belongs to $(\mathrm{fin}/S)_{\mathrm d}$. (\cite{Kato:logD} Proposition 2.1.)

(2) Let $*$ be one of d, r, e, and f. 
  Let $0 \to H' \to H \to H'' \to 0$ be an exact sequence of sheaves of abelian groups on the kfl site of $S$. 
  If $H'$ and $H''$ belong to $(\mathrm{fin}/S)_{*}$, then $H$ also belongs to $(\mathrm{fin}/S)_{*}$. (\cite{Kato:logD} Proposition 2.3.)

(3) Assume that the underlying scheme of $S$ is noetherian strictly local. 
  Let $H \in (\mathrm{fin}/S)_{\mathrm f}$. 
  Then there is a unique exact sequence $0\to H^{\text{con}} \to H \to H^{\mathrm{\acute et}}\to 0$ 
of objects of $(\mathrm{fin}/S)_{\mathrm f}$ such that kfl locally the following holds: Both 
$H^{\text{con}}$ and $H^{\mathrm{\acute et}}$ belong to $(\mathrm{fin}/S)_{\mathrm c}$, $H^{\text{con}}$ is connected, and $H^{\mathrm{\acute et}}$ is \'etale. 
(\cite{Kato:logD} Proposition 2.6.)
  Furthermore, $H$ belongs to $(\mathrm{fin}/S)_{\mathrm r}$ if and only if 
$H^{\text{con}}$ belongs to $(\mathrm{fin}/S)_{\mathrm c}$. 
(\cite{Kato:logD} Proposition 2.7 (2).)

(4) Let $H \in (\mathrm{fin}/S)_{\mathrm f}$. 
  Let $*$ be one of c, d, r, and e.
  Then $H$ belongs to $(\mathrm{fin}/S)_{*}$ if the pullback of $H$ to any closed point of $S$ belongs to $(\mathrm{fin})_{*}$.
(\cite{Kato:logD} Proposition 2.16.)
\end{para}

 \begin{prop}\label{W4} 
Let $A$ be a weak log abelian variety over an fs log scheme $S$. 
Let $H$ be a subgroup sheaf of $A[n]$ and assume that $H$ belongs to $(\mathrm{fin}/S)_{\mathrm r}$. 

$(1)$ $H$ belongs to $(\mathrm{fin}/S)_{\mathrm d}$. 

$(2)$  Let $s\in S$. \'Etale locally on $S$ around $\bar s$, $H$ has an increasing filtration $(W_wH)_w$ by subobjects in $(\mathrm{fin}/S)_{\mathrm d}$
having the following properties {\rm (i)} and {\rm (ii)}.

{\rm (i)} $W_0H=H$. $W_{-3}H=0$.  $W_{-1}H$ and $H/W_{-2}H$, and $\gr^W_wH$ for all $w\in \Z$ belong to $(\mathrm{fin}/S)_{\mathrm c}$.

{\rm (ii)} For every $r\geq 1$, if $m_{\bar s}$ denotes the maximal ideal of the strict localization $\cO_{S,\bar s}$ of $\cO_{S,s}$, the pullback of $W_wH$ on $\Spec(\cO_{S, \bar s}/m_{\bar s}^r)$ is the intersection of $H$ and $W_wA[n]$, where $W_wA[n]$ is the $W_w$ in \ref{W1}. 

$(3)$ Let $s\in S$ and let $X$ and $Y$ be abelian groups such that $X\to \bar X_{\bar s}$ and $Y\to \bar Y_{\bar s}$ are isomorphisms. 
Then \'etale locally on $S$ around $\bar s$, we have $\gr^W_0A[n]=Y/nY$ and $\gr^W_{-2}A[n]=\cH om(X/nX, \Z/n\Z(1))$.

\end{prop}

 \begin{rem}
\label{r:A[n]} 
The fact that $A[n]$ belongs to $(\mathrm{fin}/S)_{\mathrm r}$ is proved in \cite{KKN4} Proposition 18.1. 

By using it, the fact that $A[n]$ belongs to $(\mathrm{fin}/S)_{\mathrm d}$ (that is, the 
 case $H=A[n]$ of the above (1)) is proved in \cite{Kato:logD} Theorem 4.5 (1) for log abelian varieties. The proof there works for weak log abelian varieties, and is  essentially contained in the proof  below. 

\end{rem}

\begin{para} We prove Proposition \ref{W4}. 
  First, by \cite{KKN4} Proposition 9.2 (2), 
we can replace the base by an fs log scheme of finite type over $\bZ$. 
  Next, let $s \in S$ as in (2) and (3), and we replace 
the base by its strict localization $\Spec(R)$ at $s$. 
  Third, we may assume that either $n$ is invertible on the base or the residue characteristic is $p>0$ and $n$ is a power of $p$.

  We prove the case  $H=A[n]$. 
  If $R$ is a (noetherian) complete local ring, (1)--(3) for $H=A[n]$ are obtained from \ref{W1} and \ref{W2} by GAGF for finite objects.
  Then by Artin's approximation, a general case follows. 

  We consider a general $H$. 
  Let $W_w(H):=H \cap W_w(A[n])$. 
  Then it satisfies the condition (ii) of (2). 

  We prove (1). 
  If $n$ is invertible on $S$, it is  by \ref{cat} (1).
  Assume that $n$ is a power of $p$. 
  We have an exact sequence  $0\to H^{\text{con}} \to H \to H^{\mathrm{\acute et}}\to 0$ 
with $H^{\text{con}} \in (\mathrm{fin}/S)_{\mathrm c}$ 
as in \ref{cat} (3).
  Since $n$ is a power of $p$ now, $A[n]^{\mathrm{\acute et}}$ is a quotient of $A[n]/W_{-2}(A[n])$. 
  Hence it is \'etale. 
  Since $H^{\mathrm{\acute et}}\to A[n]^{\mathrm{\acute et}}$ is injective, it implies that $H^{\mathrm{\acute et}}$ is also \'etale. 
  Since $H^{\text{con}}$ and $H^{\mathrm{\acute et}}$ belong to $(\mathrm{fin}/S)_{\mathrm c}$, $H$ belongs to $(\mathrm{fin}/S)_{\mathrm {d}}$ by \ref{cat} (2). 

  We prove (i) in (2). 
If $n$ is invertible on $S$, it is easy to see because $A[n]$, $H$, and their filters are all k\'et locally constant.
  Assume that $n$ is a power of $p$. 
  First, since $\gr^W_0H \to \gr_0^WA[n]$ is injective 
and $\gr_0^WA[n]$ is \'etale, 
$\gr^W_0H$ belongs to $(\mathrm{fin}/S)_{\mathrm c}$.

  We prove $W_{-1}H$ belongs to $(\mathrm{fin}/S)_{\mathrm c}$. 
  Since it is the kernel of $H \to \gr_0^WA[n]$ and $\gr_0^WA[n]$ is \'etale, it belongs to $(\mathrm{fin}/S)_{\mathrm r}$. 
  Applying (1) to $W_{-1}H$, we see that it belongs to 
$(\mathrm{fin}/S)_{\mathrm d}$. 
  Now we apply \cite{Kato:logD} Theorem 3.1 which says that for an object $J$ of $(\mathrm{fin}/S)_{\mathrm{d}}$, $J$ belongs to $(\mathrm{fin}/S)_{\mathrm c}$ 
if and only if the  functorial canonical map $J^{\mathrm{\acute et}}(1)\to J^{\mathrm{mult}}\otimes 
\mathbb{G}_{m,\log}/\mathbb{G}_m$ is trivial.  
  This canonical map for $A[n]$ is $$A[n]^{\mathrm{\acute et}}(1)\to \gr^W_0A[n] = Y/nY$$ $$\to \cH om(X, \Z/n\Z(1)) \otimes \mathbb{G}_{m,\log}/\mathbb{G}_m= \gr^W_{-2}A[n] \otimes \mathbb{G}_{m,\log}/\mathbb{G}_m\to A[n]^{\mathrm{mult}} \otimes \mathbb{G}_{m,\log}/\mathbb{G}_m$$ in which the second arrow is induced by $X\times Y \to \mathbb{G}_{m,\log}/\mathbb{G}_m$. 
  Hence the canonical map for $W_{-1}H$ is trivial so that $W_{-1}H$ belongs to $(\mathrm{fin}/S)_{\mathrm c}$.

 Next we prove that $\gr_{-2}H=W_{-2}H$ belongs to $(\mathrm{fin}/S)_{\mathrm c}$.
 We have $W_{-2}H \subset H^{\text{mult}}$, where mult means the multiplicative part. 
 Further, $W_{-2}H$ is the kernel of $H^{\text{mult}} \to A[n]^{\text{mult}}/W_{-2}A[n]$.
 Hence it belongs to $(\mathrm{fin}/S)_{\mathrm c}$.

 Since $\gr_{-1}H = W_{-1}H/W_{-2}H$, it also belongs to $(\mathrm{fin}/S)_{\mathrm c}$.

 Finally, $\gr_{0}H, \gr_{-1}H \in (\mathrm{fin}/S)_{\mathrm c}$ and \ref{cat} (2) imply that $H/W_{-2}H$ belongs to $(\mathrm{fin}/S)_{\mathrm d}$. 
  Again by \cite{Kato:logD} Theorem 3.1, 
$H/W_{-2}H$ belongs to $(\mathrm{fin}/S)_{\mathrm c}$. 

\end{para}

\begin{para} Let the situation be as in Proposition \ref{W4} and assume that $S$ is the spectrum of a henselian local ring. 

  Note that $G[n]$ is not necessarily finite. We have $W_{-1}A[n]\subset G[n]$. This is reduced to the case of a noetherian complete local ring, and to the formal situation by the fact that $W_{-1}A[n]$ is finite.

In the case over a henselian discrete valuation ring, the above is closely related to the semi-stable reduction case (section 3) of \cite{GR}. In \cite{GR},  $W_{-1}A[n]$ is called the fixed part, and $W_{-2}A[n]$ is called the toric part. 
\end{para}

\section{The Quotient of a log abelian variety by a finite group scheme}\label{s:A/H}
  In this section, we prove the following general theorem. 

 \begin{thm}\label{A/H} Let $A$ be a log abelian variety over an fs log scheme $S$ and let $H\subset A$ be a subgroup sheaf. 
  Assume that $H$ belongs to the category $(\mathrm{fin}/S)_{\mathrm r}$.
  Then the quotient sheaf $A/H$ as a kfl sheaf is a log abelian variety.
\end{thm}

\noindent {\it Remark.}
 The case of this theorem where $S$ is an fs log point is proved in 
\cite{Z2} Proposition 3.2 (3). 

\begin{para}\label{Howto} 
  To prove the theorem, we fix a point $s$ of $S$ and consider \'etale locally on $S$ around $s$.
 
If we have an exact sequence $0\to H_1 \to H \to H_2 \to 0$ with $H_1, H_2$ being objects of $(\mathrm{fin}/S)_{\mathrm r}$, 
we are reduced to the cases $H=H_1$ and $H=H_2 \subset A/H_1$. Hence by Proposition \ref{W4},  we may assume that we are in one of the following situation (i), (ii), and (iii).
 
  (i) $H=W_{-2}H$.
 
 (ii) $H= W_{-1}H$ and $W_{-2}H=0$. 
 
 (iii) $W_{-1}H=0$.
 
  We assume that one of them holds.  
\end{para}

\begin{para}
Let $A'=A/H$ be the quotient. 
 We check that $A'$ satisfies the conditions \cite{KKN2} 4.1.1, 4.1.2, and 4.1.3 in the definition of log abelian variety.

The condition 4.1.1 
is reduced to the case of this theorem where the base is an fs log point (cf.\ the above remark after Theorem \ref{A/H}). 
\end{para}

\begin{para}
Next, we check the condition 4.1.2.
We know that $A$ satisfies this condition, that is, 
we  have an exact sequence 
$$0\to G \to A \to \cHom(\bar {X}, \Gmlog/\Gm)^{(\bar{Y})}/\bar{Y}  \to 0,$$ 
where $G$ is the semiabelian part of $A$ and $\bar X \times \bar Y \to \Gmlog/\Gm$ is the associated admissible pairing. 
  We prove that there is a similar exact sequence 
$$0\to G' \to A' \to \cHom(\bar {X'}, \Gmlog/\Gm)^{(\bar{Y'})}/\bar{Y'}  \to 0$$ by defining $G'$, $\bar{X'}\subset {\bar X}$, and $\bar{Y'}\supset \bar Y$ as follows.  

In each case in \ref{Howto}, we define $(X', Y', G')$ as follows.
Recall that we are working around $\bar s$. 
Let $X$ and $Y$ be abelian groups such that $X\to \bar X_{\bar s}$ and $Y\to \bar Y_{\bar s}$ are isomorphisms. 

 In the case (i), $X'$ is the kernel of the homomorphism $X\to \cH om(H, \mathbb{G}_m)$ induced by the embedding $H\subset \cH om(X,\mathbb{G}_m)$, $Y'=Y$, $G'=G/H$.
 
 In the case (ii), $X'=X$, $Y'=Y$, $G'=G/H$.
 
 In the case (iii), $X'=X$, $G'=G$, and $Y'=\frac{1}{n}Y''$, where $Y''\subset Y$ is the inverse image of $H$ under $Y\to Y/nY\cong A[n]$ for $n\geq 1$ which kills $H$  ($n$  can be taken locally on the base).
  
Then $G'$ is semiabelian.  

Let $\bar{X'}$ be the image of $X' \to \bar X$ and 
$\bar{Y'}$ the image of $Y' \to \bar Y_{\bQ}$.
  Then we have the desired exact sequence. 
\end{para}

\begin{para} We check the condition
4.1.3 about the diagonal morphism $A'\to A'\times A'$. Since $H$ is killed by some integer $n\geq 1$, the homomorphism $n: A \to A$ induces a surjective homomorphism $A' \to A$ of kfl sheaves with kernel $H':=A[n]/H$. 
 Note that a log abelian variety is a kfl sheaf by Proposition \ref{p:kfl}. 

Then for a morphism $U \to A'\times A'$ from an fs log scheme, 
the fiber product $V$ of $U\to A \times A \leftarrow A$ is represented by a finite fs log scheme over $U$. 
  Further, the fiber product of $V \to A'\times_A A' \leftarrow A'$ is the fiber product of $V\to H' \times A \leftarrow A$ and hence it is represented by the inverse image of $1$ under $V\to H'$, which is finite over $U$. 
  Hence the diagonal $A' \to A' \times A'$ is represented by finite morphisms. 
\end{para}

\section{Dual is a log abelian variety}
\label{s:dual}
  We recall the definition of the dual of a weak log abelian variety and prove that the dual of a polarizable log abelian variety is a log abelian variety. 

\begin{para} Let $A$ be a weak log abelian variety over an fs log scheme $S$. We define the dual $A^*$ of $A$ as a sheaf on $(\fs/S)_{\et}$ of abelian groups as follows.

  As in noted in \ref{W2}, in the case of constant degeneration, $A^*$ is already defined and is a log abelian variety. 

We define the {\it dual} $A^*$ of $A$ to be the subsheaf of $\cExt(A, \Gmlog)$ consisting of all local sections whose pullbacks to fs log points belong 
to the dual of $A$. 
(Note that $A$ becomes of constant degeneration on any fs log point.)
\end{para}

\begin{prop}\label{dual1} Let $p: A \to \cExt(A, \Gmlog)$  be a polarization. Then the image of $p$ coincides with $A^*$. 
\end{prop}

  When $p$ is a principal polarization, this together with 
\cite{KKN7} Proposition 2.8 implies that 
$A^*$ here coincides with the dual in the sense of \cite{KKN7} 3.3. 

Before proving Proposition \ref{dual1}, we show some other propositions. 

\begin{prop}
\label{p:Hom}
  Let $A$ be a log abelian variety over $S$. Then the following holds. 

$(1)$ $\cHom(A,\Gmlog)=0$.

$(2)$ $\cHom_{\ket}(A,\Gmlog)=0$.

$(3)$ $\cHom_{\kfl}(A,\Gmlog)=0$.
\end{prop}

\begin{pf}
  Since $A$ and $\Gmlog$ are kfl sheaves (Proposition \ref{p:kfl} and \cite{Kato:FI2} Theorem 3.2), (1)--(3) are equivalent and reduced to \cite{KKN5} Proposition 2.1.
\end{pf}

  The next is the log Hilbert 90, which was already sketched in the correction to \cite{KKN5} in the last part of \cite{KKNlapel}.

\begin{prop}
\label{p:Hilbert90}
  Let $X$ be an fs log scheme and let $A$ be a log abelian variety over an fs log scheme $S$.
  Then the following holds.

  $(1)$ $H^1_{\kfl}(X,\Gmlog)=H^1_{\ket}(X,\Gmlog)=H^1(X,\Gmlog)$. 

  $(2)$ $H^1_{\kfl}(A,\Gmlog)=H^1_{\ket}(A,\Gmlog)=H^1(A,\Gmlog)$. 

\end{prop}

\begin{pf}
  (1) is proved in \cite{Kato:FI2} Theorem 5.1 under the assumption that $X$ is locally noetherian.
  Since the problem is local and $H^1_{\kfl}$, $H^1_{\ket}$, and $H^1$ commute with the limit of schemes, the general case is reduced to this case. 

  (2) Since $A$ is covered by representable objects by \cite{KKN5} Proposition 11.1, it is reduced to (1). 
\end{pf}

\begin{prop}
\label{p:Ext}
  Let $A$ be a log abelian variety over an fs log scheme $S$. Then the following holds. 

$(1)$ $\cExt(A,\Gmlog)(T) = \Ext(A_T,\Gmlog)$ for any $T \in (\fs/S)$. 

$(2)$ $\cExt(A,\Gmlog)$ is a kfl sheaf.

$(3)$  As a functor, $ \cExt(A, \Gmlog)$ is locally of finite presentation. 

\end{prop}

\begin{pf}
(1) By the local-global spectral sequence, it is reduced to Proposition \ref{p:Hom} (1). 

(2) By the same argument, Proposition \ref{p:Hom} (3) implies 
$\cExt_{\kfl}(A,\Gmlog)(T) = \Ext_{\kfl}(A_T,\Gmlog)$ for any $T \in (\fs/S)$. 
  Hence it is enough to show $\Ext_{\kfl}(A_T,\Gmlog)=\Ext(A_T,\Gmlog)$. 
  By \cite{KKN5} Lemma 3.9 (2) and \cite{KKN5} Proposition 2.1, the last equality is reduced to Proposition \ref{p:Hilbert90} (2).

(3) By (1) and (2), it is reduced to \cite{KKN5} Proposition 12.8 (1).
\end{pf}

\begin{para}
  We prove Proposition \ref{dual1}. 

  We may assume that the base is of finite type over $\bZ$.
  Here we use a variant of Proposition 2.6 in \cite{KKN7} without the principality of the polarization, which is proved in the same way as it. 

  The polarization $p$ induces $A\to A^*$. 
  It is sufficient to prove the surjectivity of $A\to A^*$. 

Let $b\in A^*(U)$, where $U$ is an fs log scheme over $S$,  and we prove that $b$ belongs to the image. 
  We consider \'etale locally at a geometric point $u$ of $U$. 
  We may assume by Proposition \ref{p:Ext} (3) that $U$ 
is of finite type over $\Z$. 
  Let $(R,m_R)$ be the strict local ring at $u$ and $\hat R$ the completion of it. 
  By the theory of constant degeneration, the image $b_n$ of $b$ in $A^*(R/m_R^{n+1})$ ($n \ge0$) is the image of an element $a_n$ of $A(R/m_R^{n+1})$. 
  Since the kernel of $A \to A^*$ is finite so that there are only finite number of the choices of $a_n$, we can choose a compatible lift $(a_n)_n$.  
  Hence, by GAGF for sections of a log abelian variety (cf.\ \cite{KKN7} 2.10) and the injective GAGF for $\cExt(A,\Gmlog)$ 
(cf.\ Proposition \ref{p:Ext} and \cite{KKN6} Proposition 1.3), the image $\hat b$ of $b$ in $A^*(\hat R)$ is the image of an element $a$ of $A(\hat R)$. 
  This $a$ comes from an element $a'$ of $A(R')$ for some finitely generated subring $R'$ of $\hat R$ over $R$. 
  By the property of being locally of finite type of $A^*$, which is by Proposition \ref{p:Ext} (3), for some bigger $R'$, the images of $a'$ and $b$ coincide in $A^*(R')$. Take a homomorphism $R'\to R$ over $R$  by Artin's approximation theorem. 
  We get that $b$ comes from  the image of $a'$ in $A(R)$ under $A(R')\to A(R)$. 
  Again by Proposition \ref{p:Ext} (3), we see that the desired surjectivity.
\end{para}

\begin{prop}
\label{p:kfl_criterion}
  Let $(R,m)$ be a noetherian local ring endowed with an fs log structure. 
  Let $X$ be a finite fs log scheme over $R$. 
  If $X \otimes_R R/m^n$ is kfl over $R/m^n$ for any $n$, then $X$ is kfl over $R$.
\end{prop}

\begin{pf}
  The kummerness is reduced to that on the closed point. 
  Taking a kfl cover, we may assume that $X$ is strict over $R$ by \cite{INT} Theorem 0.2.
  Then, since $X$ is finite, its flatness is reduced to the case where $R$ is complete and then to the case where $R$ is artinian.
\end{pf}

\begin{prop}\label{dual4} The kernel of a polarization $p: A \to A^*$ belongs to $(\mathrm{fin}/S)_{\mathrm r}$. 
\end{prop}

\begin{pf} 
  First we may assume that an integer $n$ kills the kernel of $p$. 
  This is reduced to the case of constant degeneration by the same argument in \cite{KKN7} 2.10. 
  Then the kernel of $p$ is the kernel of the induced 
$A[n] \to \cExt(A,\Gmlog)[n]$.

  Next we prove that $\cExt(A,\Gmlog)[n]$ belongs to $(\mathrm{fin}/S)_{\mathrm r}$. 
  By the exact sequence 
$$0 \to A[n] \to A \overset n \to A \to 0$$
and Proposition \ref{p:Hom} (3), we see that 
$\cExt(A,\Gmlog)[n]$ is isomorphic to 
$\cHom_{\kfl}(A[n],\Gmlog)=\cHom_{\kfl}(A[n],\mu_n)$.
  By \ref{cat} (4), we may assume that the base is an fs log point. 
  Then the dual $A^*$ is a log abelian variety and 
$\cHom_{\kfl}(A[n],\mu_n)$ is isomorphic to $A^*[n]$, which belongs to 
$(\mathrm{fin}/S)_{\mathrm r}$ by Remark \ref{r:A[n]}.

  Since $A[n]$ is also belongs to $(\mathrm{fin}/S)_{\mathrm r}$ by Remark \ref{r:A[n]}, the kernel of $p:A[n] \to \cExt(A,\Gmlog)[n]$ is represented by a finite group fs log scheme. 

  We prove that it is kfl. 
  We claim that it is reduced to the case of constant degeneration. 
  In fact, 
  we may assume that $S$ is the spectrum of a local noetherian ring. 
  Then by Proposition \ref{p:kfl_criterion}, its kflness is reduced to the case where $S$ is the spectrum of an artinian local ring, and hence to the case where $A$ is of constant degeneration. 

  We assume that $A$ is of constant degeneration. 
  Then the kernel $H$ of $p$ has an increasing filtration $W$ such that $W_0H=H$, $W_{-3}H=0$, and $\gr^W_wH$ for $w=0$ (resp.\ $w=-1$, resp.\ $w=-2$) is the cokernel (resp.\ kernel, resp.\ kernel) of the associated homomorphism $p:Y\to X$ (resp.\ $p:B\to B^*$, resp.\ $p:T\to T^*$). These belong to $(\mathrm{fin}/S)_{\mathrm c}$.
Hence by \ref{cat} (2), $H$ belongs to $(\mathrm{fin}/S)_{\mathrm d}$.
\end{pf} 

  The above proof together with \ref{cat} (4) also shows the following proposition. 

\begin{prop}
\label{d}
For any $n \ge1$, the Cartier dual of $A[n] \in (\mathrm{fin}/S)_{\mathrm d}$ is $A^*[n]$.
\end{prop}

  The next is the main result of this article. 

\begin{thm}\label{dualla} For an \'etale locally polarizable log abelian variety $A$, its dual $A^*$ is a log abelian variety.
\end{thm}

\begin{pf} We may assume that $A$ is polarizable. 
  Let $p:A\to A^*$ be a polarization. This map is surjective by Proposition \ref{dual1}. Further, the kernel $H$ of this map belongs to 
$(\mathrm{fin}/S)_{\mathrm r}$ by Proposition \ref{dual4}. 
By Theorem \ref{A/H}, $A^*=A/H$ is a log abelian variety.
\end{pf}

\noindent {\it Remark.} 
When the degree of the polarization $A \to A^*$ is invertible on the base $S$, the kernel of the polarization is log \'etale. 
  Hence, the above proof of Theorem \ref{dualla} in this situation can be simplified. 

\begin{prop}
  Let $A$ be an \'etale locally polarizable log abelian variety. 
  Let $G$ and $\bar X \times \bar Y \to \Gmlog/\Gm$ be the semiabelian part and the associated admissible pairing, respectively. 
  Let $G^*$ and $\bar X^* \times \bar Y^* \to \Gmlog/\Gm$ be the semiabelian part and the associated admissible pairing for $A^*$, respectively. 
  Then the following hold.   

$(1)$ $\bar Y^*$ is identified with $\bar X$, $\bar X^*$ is identified with $\bar Y$, and $\bar X \times \bar Y \to \Gmlog/\Gm$ is identified with 
$\bar X^* \times \bar Y^* \to \Gmlog/\Gm$ after interchanging the first and the second variables. 

$(2)$ $G^*$ is identified with $\cExt(A, \Gm)$.
\end{prop}  

\begin{pf}
  (1) is reduced to the constant degeneration case. 

  (2) $G^*$ is the kernel of $\cExt(A,\Gmlog) \to \cHom(\bar Y,\Gmlog/\Gm)/\bar X$. On the other hand, $\cExt(A,\Gm)$ is the kernel of $\cExt(A,\Gmlog) \to 
\cExt(A,\Gmlog/\Gm)$ by \cite{KKN2} Lemma 6.1.1. 
  Hence (2) is reduced to the injectivity of the composite 
$\cHom(\bar Y,\Gmlog/\Gm)/\bar X \to \cExt(A/G,\Gmlog/\Gm) \to 
\cExt(A,\Gmlog/\Gm)$. 
  The first homomorphism is injective by \cite{KKN5} Lemma 2.6. 
  The second homomorphism is also injective by \cite{KKN2} Lemma 6.1.1.
\end{pf}

\section{$A^{**}=A$}
\label{s:duality}

We prove that the dual of the dual of $A$ is identified with $A$. 

\begin{prop}  
  Let $A$ be an \'etale locally polarizable log abelian variety. 
  Then there is a natural isomorphism $(A^*)^*=A$. 
\end{prop}

\begin{pf}
  First we have a homomorphism 
$A\to \cExt(\cExt(A, \Gmlog), \Gmlog)\to \cExt(A^*, \Gmlog)$, where the first arrow exists formally and the second arrow is by $A^*\subset \cExt(A, \Gmlog)$. 
  The composite map induces a homomorphism $A\to (A^*)^*$, which is seen by reducing to the case where the base is an fs log point. 
  We prove that this $A \to (A^*)^*$ is an isomorphism. 
  We may assume that $A$ is polarizable. 
  Then, by the same argument as in the non-log case, $A^*$ is polarizable.
  Hence, by Theorem \ref{dualla}, $(A^*)^*$ is a log abelian variety. 
  Therefore, the problem is reduced to the case of constant degeneration by the next proposition. 
\end{pf}

\begin{prop}
A homomorphism $f\colon A \to A'$ 
of log abelian varieties over an fs log scheme $S$ 
is an isomorphism if the pullback to any 
subscheme of $S$ where the log structure is constant 
is an isomorphism. 
\end{prop}

\begin{pf}
  Let $G$ and $G'$ be the semiabelian parts of $A$ and $A'$, respectively. 
  The given homomorphism $f$ induces $G \to G'$ and $A/G \to A'/G'$. 
  By assumption, other induced homomorphisms 
$\overline Y \to \overline Y'$ 
and 
$\overline X' \to \overline X$ 
are isomorphisms (under the usual notation). 
  Since $A/G \to A'/G'$ is isomorphic to 
$\cHom(\overline X,\Gmlog/\Gm)^{(\overline Y)}/\overline Y
\to \cHom(\overline X',\Gmlog/\Gm)^{(\overline Y')}/\overline Y'$, 
it is also an isomorphism. 
  Because $A'$ is a log abelian variety, by \cite{KKN5} Section 10, there is a 
complete and wide fan $\Sig'$, which is regarded as a subsheaf of $A'/G'$. 
  By the pullback with respect to the isomorphism 
$A/G \to A'/G'$, we have a subsheaf $\Sig$ of $A/G$. 
Then \cite{KKN6} Proposition 4.7 implies that $f$ induces 
a morphism of proper models 
$(P,\Sig,\ldots)$ to $(P',\Sig',\ldots)$. 
  By assumption, the morphism $P \to P'$ is an isomorphism on all constant log loci.
  Since $P$ and $P'$ are proper, $P \to P'$ is an isomorphism.
  This implies that $G \to G'$ is also an isomorphism because $G$ and $G'$ 
are the pullbacks of the unit section of $P/G$ and $P'/G'$, respectively. 
  Therefore, the original $f$ is an isomorphism. 
\end{pf}

\noindent Takeshi Kajiwara

\noindent Department of Applied mathematics \\
Faculty of Engineering \\
Yokohama National University \\
Hodogaya-ku, Yokohama 240-8501 \\
Japan

\noindent kajiwara@ynu.ac.jp
\par\bigskip\par

\noindent Kazuya Kato

\noindent 
Department of Mathematics
\\
University of Chicago
\\
5734 S.\ University Avenue
\\
Chicago, Illinois, 60637 \\
USA

\noindent kkato@math.uchicago.edu
\par\bigskip\par

\noindent Chikara Nakayama

\noindent Department of Economics \\ Hitotsubashi University \\
2-1 Naka, Kunitachi, Tokyo 186-8601 \\ Japan

\noindent c.nakayama@r.hit-u.ac.jp
\end{document}